\documentclass[10pt]{article}
\usepackage{mathrsfs}
\usepackage{amsfonts}
\usepackage{amsthm,amsmath,amssymb,anysize}
\usepackage[all]{xy}
\newtheorem{lemma}{Lemma}[section]
\newtheorem{theorem}[lemma]{Theorem}

\newtheorem{proposition}[lemma]{Proposition}
\newtheorem{definition}[lemma]{Definition}
\newtheorem{corollary}[lemma]{Corollary}

\usepackage[numbers,sort&compress]{natbib}
\marginsize{3.5cm}{3.5cm}{3.6cm}{3cm}
\numberwithin{equation}{section}

\newcommand{\Z}{\mathrm{Z}}
\title{\textsf{On $\mathrm{ID}^{*}$-superderivations of Lie superalgebras }\footnote{Supported by the NSF
  of China (11471090, 11701158)}}
\author{\textsc{\textsc{Wende Liu}$^{1,}$ and \textsc{Mengmeng Cai}$^{2,}$\footnote{Correspondence: 1162823157@qq.com}}
\\
\\
\textit{$^1$School of Mathematics and Statistics,}
\textit{Hainan Normal University,}\\\textit{Haikou 571158, P. R. China}\\
 \textit{$^2$School of Mathematical Sciences},
  \textit{Harbin Normal University} \\
  \textit{Harbin 150025, P. R. China}
  }
\date{}

\begin{document}
\maketitle
\begin{quotation}
\small\noindent\textbf{Abstract}: Let $L$ be a Lie superalgebra over a  field of characteristic different from $2,3$ and write  $\mathrm{ID}^{*}(L)$ for the Lie superalgebra consisting of superderivations mapping $L$ to $L^{2}$ and the central elements to zero. In this paper we first give an upper bound for the superdimension of $\mathrm{ID}^{*}(L)$ by means of linear vector space decompositions. Then we characterize the $\mathrm{ID}^{*}$-superderivation superalgebras for the nilpotent Lie superalgebras of class 2 and the model filiform Lie superalgebras by methods of block matrices.

\vspace{0.2cm} \noindent{\textbf{Keywords}}:  $\mathrm{ID}^{*}$-superderivation; nilpotent Lie superalgebras of class 2; model filiform Lie superalgebras

\vspace{0.2cm} \noindent{\textbf{Mathematics Subject Classification 2010}}: 17B05, 17B30, 17B56
\end{quotation}

\setcounter{section}{0}

\section{Introduction and preliminaries}

As is well known, derivations and their various generalizations are important topics in the Lie (super)algebra theory (see \cite{tang}, for example). In 2015, H. Arabynai and F. Saeedi studied derivation algebras of Lie algebras and introduced the notion of $\mathrm{ID}^{*}$-derivations for Lie algebras \cite{E}. An $\mathrm{ID}^{*}$-derivation of a Lie algebra $L$ is a derivation  sending $L$ into $L^2$ and $\mathrm{Z}(L)$ to zero, where $L^2$ and $\mathrm{Z}(L)$ are the derived algebra and the center of $L$, respectively. All the $\mathrm{ID}^{*}$-derivations of a Lie algebra constitute a subalgebra of the full derivation algebra.

The notion of $\mathrm{ID}^{*}$-derivations may be naturally generalized to  Lie superalgebra case.  In this paper, we first  give an upper bound for the superdimension of the  $\mathrm{ID}^{*}$-superderivation superalgebra for a Lie superalgebra $L$, in terms of the superdimension  of $L^{2}$  and  the minimal generator number pair (see Definition \ref{d3}) of $L/\mathrm{Z}(L)$, where $L^2$ and $L/\mathrm{Z}(L)$ are the derived superalgebra and the central quotient of $L$, respectively. Then we show that the minimal generator number pairs  are unique for finite-dimensional nilpotent Lie superalgebras. Finally, we characterize the $\mathrm{ID}^{*}$-superderivation superalgebras of the nilpotent Lie superalgebras of class 2 and the model filiform Lie superalgebras in terms of block matrices, which  tell us a fact  that the upper bound which we have obtained for the superdimension of the $\mathrm{ID}^{*}$-superderivation superalgebra for a Lie superalgebra   is sharp.

Let $\mathbb{F}$ be the ground field of  characteristic different from $2, 3$ and $\mathbb{Z}_{2}:=\{\bar{0}, \bar{1}\}$   the abelian group of order $2$. For a homogeneous element $x$ in a vector superspace  $V=V_{\bar{0}}\oplus V_{\bar{1}}$, write $|x|$ for the
parity of $x$. In this paper the symbol $|x|$ implies that $x$ has been assumed to be a homogeneous element.

In $\mathbb{Z} \times \mathbb{Z}$, we define a partial order as follows:
$$(m, n)\leq(k, l) \Longleftrightarrow m\leq k, n\leq l.$$
For $m,n\in \mathbb{Z}$, we write $|(m,n)|=m+n.$ We also view $\mathbb{Z} \times \mathbb{Z}$ as an additive group in the usual way.

Write $\mathrm{sdim} V$ for the superdimension of a vector superspace $V$ and $\dim V$ for the dimension of  $V$ as an ordinary vector space. Then
$\dim V=|\mathrm{sdim} V|.$

A linear map  of parity $\alpha\in \mathbb{Z}$, $D: L\rightarrow L$, is said to be a superderivation of $L$, if
$$D[x, y]=[D(x), y]+(-1)^{\alpha|x|}[x, D(y)]$$ for all $x, y \in L$.
Denote by $\mathrm{Der}_{\alpha}(L)$ the set of all the superderivations of parity $\alpha$ of $L$, where $\alpha\in \mathbb{Z}_{2}$.
Then the superspace $\mathrm{Der}(L):=\mathrm{Der}_{\bar{0}}(L) \oplus \mathrm{Der}_{\bar{1}}(L)$ is a Lie superalgebra with respect to the usual bracket
$$[D, E]=DE - (-1)^{|D||E|}ED, $$
where $D, E \in \mathrm{Der}(L) $. The elements of $\mathrm{Der}(L)$ are called superderivations of $L$ and $\mathrm{Der}(L)$ is called the superderivation superalgebra of $L$.
For $x\in L$, the map $\mathrm{ad}x: L\rightarrow L $ given by $y\longmapsto [x, y]$ is a superderivation of $L$, which is said to be inner. The set of all inner superderivations of $L$ is denoted by $\mathrm{ad}(L)$. It is a standard fact that $\mathrm{ad}(L)$ is an superideal of $\mathrm{Der}(L)$.

As in Lie algebra case \cite{E}, a superderivation of Lie superalgebra $L$ is called an $\mathrm{ID}$-superderivation if it maps $L$ to the derived subsuperalgebra.
Denote by $\mathrm{ID}(L)$ the set of all $\mathrm{ID}$-superderivations of $L$, that is,
$$\mathrm{ID}(L)=\{\alpha \in \mathrm{Der}(L) \mid \alpha(L) \in L^{2}\}.$$
Hereafter,  write $L^2$ for the derived subsuperalgebra $[L, L]$.
Denote by $\mathrm{ID}^{*}(L)$ the set of all $\mathrm{ID}$-superderivations mapping all central elements to $0$, that is,
$$\mathrm{ID}^{*}(L)=\{\alpha \in \mathrm{ID}(L)\mid \alpha(\mathrm{Z}(L))=0\}.$$
Obviously, we have a sequence of subsuperalgebras:
$$\mathrm{ad}(L) \leq  \mathrm{ID}^{*}(L) \leq \mathrm{ID}(L) \leq \mathrm{Der}(L).$$

\section{An upper bound for the superdimension of $\mathrm{ID}^{*}(L)$}

To describe the upper bound for the superdimension of the $\mathrm{ID}^{*}$-superderivation superalgebra  for a given Lie superalgebra, we need the concept of minimal  generator number pairs for a Lie superalgebra.
Write $\{x_{1}, \ldots, x_{p} \mid y_{1}, \ldots, y_{q}\}$ implying that $x_{i}$ is even and $y_{j}$ is odd in a superspace.
\begin{definition}\label{d3}
A  homogeneous generator set of a Lie superalgebra $L$,
$$
\{x_{1}, \ldots, x_{p} \mid y_{1}, \ldots, y_{q}\},
$$
is said to be minimal  if $L$ can not be generated by any  subset of $L$,
$$
\{a_{1}, \ldots, a_{s} \mid b_{1}, \ldots, b_{t}\}
$$
with $(s, t)<(p, q)$. In this case, $(p, q)$ is called a minimal  generator number pair of $L$.
\end{definition}

For finite-dimensional nilpotent Lie superalgebras, the minimal  generator number pairs are unique (see Proposition \ref{pro1}).
 However, this fact does not necessarily  hold in general.

For a subsuperalgebra $K$ of $L$ and a pair of nonnegative integers $(p, q)$, write $\mathrm{\lambda}(K; p, q)$ for the number pair
$$
(p\cdot{\dim}K_{\bar{0}}+q\cdot{\dim}K_{\bar{1}}, q\cdot{\dim}K_{\bar{0}}+p\cdot{\dim}K_{\bar{1}}).
$$

The following theorem gives an upper bound for the superdimension of $\mathrm{ID}^{*}(L)$.

\begin{theorem}\label{c}
Suppose $L$ is a Lie superalgebra such that ${\dim}L^{2}<\infty$ and $L/\mathrm{Z}(L)$ is finitely generated. Then
$$
\mathrm{sdim}\mathrm{ID}^{*}(L)\leq \mathrm{\lambda}(L^{2}; p, q),
$$
where $(p, q)$ is a minimal generator number pair of $L/\mathrm{Z}(L)$. In particular, $\mathrm{ID}^{*}(L)$ is finite-dimensional.
\end{theorem}

\begin{proof} Suppose $\{x_{1}, \ldots, x_{p} \mid y_{1}, \ldots, y_{q}\}$ is a subset of $L$ such that
$$
\{x_{1}+\mathrm{Z}(L), \ldots, x_{p}+\mathrm{Z}(L) \mid y_{1}+\mathrm{Z}(L), \ldots, y_{q}+\mathrm{Z}(L)\}
$$
is a minimal generator set of $L/\mathrm{Z}(L)$.
Define
\begin{eqnarray*}
\phi: \mathrm{ID^{*}}_{\bar{0}}(L) &\longrightarrow& \underbrace{(L^{2})_{\bar{0}} \oplus \cdots \oplus (L^{2})_{\bar{0}}}_{\text{$p$ copies}} \oplus \underbrace{(L^{2})_{\bar{1}} \oplus \cdots \oplus (L^{2})_{\bar{1}}}_{\text{$q$ copies}}\\
\alpha &\longmapsto& \left(\alpha(x_{1}), \ldots, \alpha(x_{p}), \alpha(y_{1}), \ldots, \alpha(y_{q})\right).
\end{eqnarray*}
 Clearly,   $\phi$ is an injective linear map. Therefore
\begin{equation}\label{eqliu1049}
 \mathrm{dimID^{*}_{\bar{0}}}(L)\leq p\cdot{\dim}(L^{2})_{\bar{0}}+q\cdot{\dim}(L^{2})_{\bar{1}}.
 \end{equation}
 Similarly,
\begin{eqnarray*}
\varphi: \mathrm{ID^{*}}_{\bar{1}}(L) &\longrightarrow& \underbrace{(L^{2})_{\bar{0}} \oplus \cdots \oplus (L^{2})_{\bar{0}}}_{\text{$q$ copies}} \oplus \underbrace{(L^{2})_{\bar{1}} \oplus \cdots \oplus (L^{2})_{\bar{1}}}_{\text{$p$ copies}}\\
\beta &\longmapsto& (\beta(y_{1}), \ldots, \beta(y_{q}), \beta(x_{1}), \ldots, \beta(x_{p}))
\end{eqnarray*}
is an injective linear map and then
\begin{equation}\label{eqliu1050}
{\dim}\mathrm{ID}^{*}_{1}(L)\leq q\cdot{\dim}(L^{2})_{\bar{0}}+p\cdot{\dim}(L^{2})_{\bar{1}}.
\end{equation}
It follows from (\ref{eqliu1049}) and (\ref{eqliu1050}) that
$
\mathrm{sdim}\mathrm{ID}^{*}(L)\leq\mathrm{\lambda}(L^{2}; p, q).
$
\end{proof}

\begin{corollary}\label{tonggou}
Let $L$ be a Lie superalgebra. Then $\mathrm{ad}(L)$ is  finite-dimensional if and only if $\mathrm{ID^{*}}(L)$ is finite-dimensional.
\end{corollary}

\begin{proof} One direction is obvious. Let us suppose $\mathrm{ad}(L)$ is finite-dimensional. Then $L/\mathrm{Z}(L)\cong \mathrm{ad}(L)$  is finite-dimensional and so is $L^{2}$.
It follows from Theorem \ref{c} that $\mathrm{ID^{*}}(L)$ is finite-dimensional.
\end{proof}

\section{Nilpotent Lie superalgebras of class 2 and model filiform Lie superalgebras}

Let $L$ be a  Lie superalgebra over $\mathbb{F}$. Recall that the lower central series of $L$ is the sequence of superideals of $L$ defined inductively by $L^{1}=L$ and $L^{n}=[L^{n-1}, L]$ for $n\geq 2$.
If there exists $n\geq 2$ such that $L^{n}=0$, then $L$ is called a nilpotent Lie superalgebra.
The least integer $c$ for which $L^{c+1}=0$ is called the  class of $L$. Clearly, a Lie superalgebra is of class 1 if and only if it is abelian.

We also recall that the upper central series of $L$ is the sequence of superideals of $L$, $\Z_{i}(L)$, defined inductively:
$$\Z_{0}(L):=0,\;\Z_{i}(L)/\Z_{i-1}(L)=\Z(L/\Z_{i-1}(L))$$ for all $i\geq 1.$ Note that for $i\geq 1$,
\begin{align*}
\Z_{i}(L)&=\{x\in L\mid [x, L]\subset \Z_{i-1}(L)\}\\
&=\{x\in L\mid [[[x,\underbrace{L],L],\ldots,L}_{i\, \text{copies}}]=0\}.
\end{align*}
It is easy to see that a Lie superalgebra $L$ is of class $c$ if and only if
\begin{align*}
\Z_{c}(L)=L\; \text{and}\; \Z_{c-1}(L)\neq L.
\end{align*}
Moreover, for a nilpotent Lie superalgebra, the lower and upper central series have the same length.

We also use the notion of super-nilindex for a Lie superalgebra $L=L_{\bar{0}}\oplus L_{\bar{1}}$.
Write
$$
\mathcal{C}^{0}(L_{\alpha})=L_{\alpha},\ \mathcal{C}^{k+1}(L_{\alpha})=[L_{\bar{0}}, \mathcal{C}^{k}(L_{\alpha})],
$$
where $\alpha\in \mathbb{Z}_{2}$ and $k\geq 0$. If $L$ is nilpotent,  a pair $(p, q)$ of nonnegative integers is called the  super-nilindex of $L$, if
$$
\mathcal{C}^{p}(L_{\bar{0}})=0,\ \mathcal{C}^{p-1}(L_{\bar{0}})\neq 0;\ \mathcal{C}^{q}(L_{\bar{1}})=0,\ \mathcal{C}^{q-1}(L_{\bar{1}})\neq 0.
$$
Clearly, a Lie superalgebra  is of super-nilindex $(1,1)$ if and only if it is abelian with nontrivial even and odd parts.

A nilpotent Lie superalgebra of superdimension $(n+1, m)$, where  $n$ and $m$ are positive integers such that $n+m>2$, is said to be of  filiform if its super-nilindex is $(n, m)$. See \cite{I}, for example.

Let $L^{n,m}$ be a Lie superalgebra with  a homogeneous basis
\begin{equation}\label{q2}
\{x_{0}, \ldots, x_{n}\mid y_{1}, \ldots, y_{m}\}
\end{equation}
and multiplication given by
$$
[x_{0}, x_{i}]=x_{i+1},\ 1\leq i\leq n-1,\ [x_{0}, y_{j}]=y_{j+1},\ 1\leq j\leq m-1.
$$
Then $L^{n,m}$ is a filiform Lie superalgebra, which is called the model filiform Lie superalgebra of super-nilindex  $(n, m)$. For further information, see \cite{I}, for example. We should mention that any filiform Lie superalgebra is a deformation of a model one in some sense. See \cite{K} for more details.

Among all the nilpotent Lie superalgebras of class 2, very interesting ingredients are  the so-called generalized Heisenberg Lie superalgebras. See \cite{J} for the non-super case.
\begin{definition}\label{dd}
A  nonzero Lie superalgebra $H$ is called a generalized Heisenberg Lie superalgebra if $H^{2}=\mathrm{Z}(H)$.
\end{definition}

We note that a  Heisenberg Lie superalgebra  is a generalized Heisenberg Lie superalgebra  with a 1-dimensional center. See \cite{G} for more details.
Before considering   $\mathrm{ID}^{*}$-superderivation superalgebras of generalized Heisenberg Lie superalgebras, let us point out that  the minimal generator number pairs of nilpotent Lie superalgebras are unique. To prove this fact, we need a lemma as in the Lie algebra case. See \cite[Corollary 2]{H} and \cite[Lemma 2.1]{F}.

\begin{lemma}\label{lemmaliu2} Suppose $L$ is a finite-dimensional nilpotent Lie superalgebra.
Suppose $K$ is a subsuperalgebra of $L$ such  that $K+L^2=L$. Then $K=L$.\qed
\end{lemma}
\begin{proposition}\label{pro1}
Suppose $L$ is a finite-dimensional
 nilpotent Lie superalgebra. Then $\mathrm{sdim}(L/L^{2})$ is the unique minimal generator number pair of $L$.
\end{proposition}
\begin{proof}
 Suppose $(p, q)$ is a minimal generator number pair of $L$. Since $L/L^{2}$ is abelian, we have $(p, q)\geq \mathrm{sdim}(L/L^{2})$.  Suppose $\mathrm{sdim}L/L^{2}=(s,t)$ and $(x_{1},\ldots, x_s\mid y_1,\ldots,y_t)$ is a tuple of homogeneous elements in $L$ such that
 $(x_{1}+L^2,\ldots, x_s+L^2\mid y_1+L^2,\ldots,y_t+L^2)$ is a basis of $L/L^{2}$. If follows from Lemma \ref{lemmaliu2} that $\{x_{1},\ldots, x_s, y_1,\ldots,y_t\}$ generates the whole $L$. Consequently, $(p, q)= \mathrm{sdim}(L/L^{2}).$ The proof is complete.
\end{proof}

As an application of Proposition \ref{pro1}, we can characterize a family of  finite-dimensional nilpotent Lie superalgebras:
\begin{corollary} Suppose $L$ is a finite-dimensional nilpotent Lie superalgebra and ${\dim} L/L^{2}=1$. Then $L$ is isomorphic to one of the following Lie superalgebras:

(1) 1-dimensional Lie superalgebra with   trivial odd part;

(2) 1-dimensional Lie superalgebra with  trivial even part;

(3) 2-dimensional Heisenberg Lie superalgebra of even center, $H(0,1)$,  having a standard basis $\{z\mid w\}$ with multiplication given by $[w, w]=z$.
\end{corollary}
\begin{proof}
We have $\mathrm{sdim} L/L^{2}=(1,0)$ or $(0,1)$. By Proposition \ref{pro1}, if $\mathrm{sdim} L/L^{2}=(1,0)$, then $L$ is a 1-dimensional Lie superalgebra with   trivial odd part. Suppose $\mathrm{sdim} L/L^{2}=(0,1)$. By  Proposition \ref{pro1} again, $L$ is generated by one odd element $w$. Since $\mathrm{char}F\neq 3,$ we have $[w,[w,w]]=0$.
 Then our corollary follows from  a direct argument.
\end{proof}

Now we are in position to determine the $\mathrm{ID}^{*}$-superderivation superalgebras of the generalized Heisenberg Lie superalgebras.
\begin{proposition}\label{le1}
Suppose $H$ is a generalized Heisenberg Lie superalgebra of superdimension
$
\mathrm{sdim}H=(m, n)$ and
$\mathrm{sdim}\mathrm{Z}(H)=(m_{1}, n_{1}).
$
Then

(1) $\mathrm{ID}^{*}(H)$ is isomorphic to the Lie superalgebra consisting of  matrices
$$
\left(
  \begin{array}{cc|cc}
    0    & 0 & 0  & 0 \\
    A    & 0 & C &  0 \\\hline
    0 &   0 & 0 & 0 \\
   D & 0 & B   & 0\\
  \end{array}
\right)\in \frak{gl}(m|n),
$$
where $A, B, C$ and $D$ are arbitrary matrices of formats $m_{1}\times (m-m_{1})$, $n_{1}\times (n-n_{1})$, $n_{1}\times (m-m_{1})$  and $m_{1}\times (n-n_{1})$, respectively.

(2) $\mathrm{sdim}\mathrm{ID}^{*}(H)$ attains the upper bound $\lambda(H^{2}; p, q)$, where $(p, q)$ is the minimal generator number pair of $H/\mathrm{Z}(H)$, which coincides  with the superdimension of $H/\mathrm{Z}(H)$.

\end{proposition}
\begin{proof}
Suppose
\begin{equation}\label{eq53}
\{x_{1}, \ldots, x_{m-m_{1}}, x_{m-m_{1}+1}, \ldots, x_{m} \mid y_{1}, \ldots, y_{n-n_{1}}, y_{n-n_{1}+1}, \ldots, y_{n}\}
\end{equation}
is a basis of $H$ such that $\mathrm{Z}(H)$ is spanned by
$$
\{x_{m-m_{1}+1}, \ldots, x_{m} \mid y_{n-n_{1}+1}, \ldots, y_{n}\}.
$$
Clearly, an even linear transformation of $H$  is an $\mathrm{ID}^{*}$-superderivation if and only if its matrix with respect to  basis (\ref{eq53}) is of  form
$$
\left(
  \begin{array}{cc|cc}
    0    & 0 & \   & \ \\
    A    & 0 & \ &  \ \\\hline
    \ &   \ & 0 & 0 \\
   \ & \ & B   & 0\\
  \end{array}
\right),
$$
where $A$  is an arbitrary $m_{1}\times (m-m_{1})$ matrix  and  $B$ is an  arbitrary $n_{1}\times (n-n_{1})$ matrix. In particular,
$$
\mathrm{dimID^{*}_{\bar{0}}}(H)=mm_{1}+nn_{1}-m_{1}^{2}-n_{1}^{2}.$$
Similarly, an odd linear transformation of $H$ is an $\mathrm{ID}^{*}$-superderivation if and only if its matrix with respect to  basis (\ref{eq53}) is of  form
$$
\left(
  \begin{array}{cc|cc}
    \   & \ & 0    & 0 \\
    \  & \ & C & 0 \\ \hline
    0 &   0& \ & \ \\
   D & 0 & \  & \ \\
  \end{array}
\right),
$$
where
$C$ is an arbitrary $n_{1}\times (m-m_{1})$ matrix and $D$ is an arbitrary $m_{1}\times (n-n_{1})$ matrix.
Therefore, $$\mathrm{dimID^{*}_{\bar{1}}}(H)=nm_{1}+mn_{1}-2n_{1}m_{1}.$$
Since $H/\mathrm{Z}(H)$ is abelian and $\mathrm{sdim}(H/\mathrm{Z}(H))=(m-m_{1}, n-n_{1})$, we have $(p,q)=(m-m_{1}, n-n_{1})$ is the minimal generator number pair of $H/\mathrm{Z}(H)$. Hence
\begin{equation*}
\mathrm{\lambda}(H^{2}; m-m_{1}, n-n_{1})=(mm_{1}+nn_{1}-m_{1}^{2}-n_{1}^{2}, nm_{1}+mn_{1}-2n_{1}m_{1}).
\end{equation*}
 It follows that
$$
\mathrm{sdim}\mathrm{ID}^{*}(H)=\mathrm{\lambda}(H^{2}; m-m_{1}, n-n_{1}).
$$
The proof is complete.
\end{proof}

For further information on  Heisenberg Lie superalgebras, the reader is referred to \cite{M}.
Let us consider the $\mathrm{ID}^{*}$-superderivation superalgebras of nilpotent Lie superalgebras of class 2.
As in Lie algebra case, a nilpotent Lie superalgebra of class 2 is a direct sum of a generalized Heisenberg Lie superalgebra and an abelian Lie superalgebra \cite{J}:
\begin{proposition}\label{f}
Let $L$ be a finite-dimensional nilpotent Lie superalgebra of  class $2$. Then
\begin{equation}\label{stdcom}
L=H\oplus S,
\end{equation} where $H$ is a generalized Heisenberg Lie subsuperalgebra and $S$ is a central superideal of $L$.
\end{proposition}
\begin{proof}   Obviously, $L^{2}\subset \mathrm{Z}(L)$ and there exists a central superideal $S$ such that $\mathrm{Z}(L)=L^{2} \oplus S$. Since $L/L^{2}$ is abelian, there exists an superideal $H$ of $L$ containing $L^{2}$ such that $L/L^{2}=H/L^{2}\oplus \mathrm{Z}(L)/L^{2}$.
Then
\begin{eqnarray*}
L=H+\mathrm{Z}(L)
=H+L^{2}+S
=H+S.
\end{eqnarray*}
Note that $H\cap \mathrm{Z}(L)=L^{2}$ and $ L^{2}=H^{2}$. Since $S\cap L^{2}=0$, we have
$$
S\cap H=S\cap H\cap \mathrm{Z}(L)=S\cap L^{2}=0.
$$
Hence $L= H\oplus S$. We claim that $\mathrm{Z}(H)=L^{2}$. In fact, it is clear that $L^{2}\subset \mathrm{Z}(H)$. On the other hand, since
\begin{eqnarray*}
[\mathrm{Z}(H), L]=[\mathrm{Z}(H),  H+\mathrm{Z}(H)]=0,
\end{eqnarray*}
we have $\mathrm{Z}(H)\subset \mathrm{Z}(L)\cap H=L^{2}$. Hence $\mathrm{Z}(H)=L^{2}=H^{2}$ and then $H$ is a generalized Heisenberg Lie superalgebra. The proof is complete.
\end{proof}

  We call (\ref{stdcom}) a  standard composition for a nilpotent Lie superalgebra of class 2.
 Note that the superdimension of $L/\mathrm{Z}(L)$ is just the unique minimal generator number pair of $L/\mathrm{Z}(L)$, since $L/\mathrm{Z}(L)$ is abelian.
\begin{theorem}\label{tonggou}
Let $L$ be a finite-dimensional Lie superalgebra of nilpotent class $2$ and $(p, q)$  the minimal generator number pair of $L/\mathrm{Z}(L)$. Suppose $L=H\oplus S$ is a standard decomposition (Proposition \ref{f}) and
$$
\mathrm{sdim}H=(m, n), \ \mathrm{sdim}\mathrm{Z}(H)=(m_{1}, n_{1}), \ \mathrm{sdim}S=(s, t).
$$
Then

(1) $\mathrm{ID}^{*}(L)$ is isomorphic to the Lie superalgebra consisting of  matrices
$$
\left(
  \begin{array}{cc|cc}
    0    & 0 & 0  & 0 \\
    A    & 0 & C &  0 \\
    0    & 0 & 0  & 0 \\\hline
    0 &   0 & 0 & 0 \\
   D & 0 & B   & 0\\
    0    & 0 & 0  & 0 \\
  \end{array}
\right)\in \frak{gl}(m|n),
$$
where $A, B, C$ and $D$ are arbitrary matrices of formats $m_{1}\times (m-m_{1})$, $n_{1}\times (n-n_{1})$, $n_{1}\times (m-m_{1})$  and $m_{1}\times (n-n_{1})$, respectively.

(2) $\mathrm{sdim}\mathrm{ID}^{*}(L)$ attains  the upper bound $\lambda(L^{2}; p, q)$.
\end{theorem}

\begin{proof}
Since $L=H\oplus S$ is a standard decomposition of $L$, we have
$$
L^{2}=H^{2}, \ \mathrm{Z}(L)=\mathrm{Z}(H)\oplus S.
$$
Therefore, one sees that $\mathrm{ID}^{*}(L)\cong \mathrm{ID}^{*}(H)$.  Then our theorem follows from Proposition \ref{le1}.
\end{proof}

Finally, let us consider the $\mathrm{ID}^{*}$-superderivation superalgebras of model filiform Lie superalgebras.
By Proposition \ref{pro1}, we have the following corollary.

\begin{corollary}\label{coro1}
The minimal generator number pair of model filiform Lie superalgebra $L^{n, m}$ is $(2, 1)$.\qed
\end{corollary}

\begin{theorem}\label{th1}
Let $L=L^{n, m}$ be a model filiform Lie superalgebra.

(1) If $m> n=1$, then $\mathrm{ID}^{*}(L)$ is isomorphic to the Lie superalgebra consisting of  matrices
$$
\left(
  \begin{array}{cc|cc}
    \    & \ & \  & \ \\
    \    & \ & \  & \ \\\hline
    0 &   0 & 0 & 0 \\
   D & 0 & B   & 0\\
  \end{array}
\right)\in \frak{gl}(m|n),
$$
where $B$ is of form
\begin{equation}\label{matricesb}
\left(
  \begin{array}{cccc}
    b_{1,1}  & \ & \  &\     \\
    b_{2,1}  & b_{1,1} & \  &\     \\
    \vdots & \vdots &  \ddots& \  \\
   b_{m-1,1}& b_{m-2,1}  & \cdots  &  b_{1,1}\\
  \end{array}
\right),
\end{equation} $D$ is of form
\begin{equation}\label{matricesd1}
\left(
  \begin{array}{c}
    d_{1,1}  \\
    \vdots \\
   d_{m-1,1} \\
  \end{array}
\right)
\end{equation}
with  $b_{ij}$, $d_{kl}$  being arbitrary elements in $\mathbb{F}$.

If $n> m=1$, then $\mathrm{ID}^{*}(L)$ is isomorphic to the Lie superalgebra consisting of  matrices
$$
\left(
  \begin{array}{cc|cc}
      0 &0  & \    & \ \\
      A& 0 & \ & \ \\\hline
     \ &  \  & \ &  \\
   \ & \ & \  &   \\
  \end{array}
\right) \in \frak{gl}(m|n)     ,
$$
where $A$ is of form
\begin{equation}\label{matricesa}
\left(
  \begin{array}{ccccc}
    a_{1,1}  & a_{1,2}& \    & \ & \ \\
     a_{2,1}  & a_{2,2}& a_{1,2}& \ &\ \\
   \vdots &  \vdots& \vdots& \ddots & \  \\
   a_{n-1,1}& a_{n-1,2}& a_{n-2,2}& \cdots &a_{1,2} \\
  \end{array}
\right) \end{equation}
with $a_{ij}$ being  arbitrary elements in $\mathbb{F}$.

If $m\geq n\geq 2$, then $\mathrm{ID}^{*}(L)$ is isomorphic to the Lie superalgebra consisting of  matrices
$$
\left(
  \begin{array}{cc|cc}
      0 &0  & 0    & 0 \\
      A& 0 & C & 0 \\\hline
     0 &  0  & 0 & 0 \\
   D & 0 & B  &  0 \\
  \end{array}
\right)\in \frak{gl}(m|n)      ,
$$
where $A$ is of form (\ref{matricesa}), $B$ is of form (\ref{matricesb}), $D$ is of form
\begin{equation}\label{matricesd3}\left(
  \begin{array}{ccccc}
    d_{1,1}  & d_{1,2}   & \ & \ \\
    \vdots &  \vdots & \ddots &\ \\
   d_{n-1,1}& d_{n-1,2}  & \cdots  &d_{1,2}\\
    \vdots &  \vdots & \ &\vdots \\
   d_{m-1,1}& d_{m-1,2}& \cdots  &d_{m-n+1,2}\\

  \end{array}
\right)
\end{equation}
and $C$ is of form
\begin{equation}\label{matricesc3}\left(
  \begin{array}{cccc}
    c_{1,1}  & \ & \  &\     \\
    c_{2,1}  & c_{1,1}& \  &\  \\
    \vdots &  \vdots&  \ddots& \ \\
   c_{n-1,1}&  c_{n-2,1}&\cdots    &c_{1,1}\\
  \end{array}
\right)
\end{equation}
with $c_{st}$ and $d_{mn}$ being arbitrary elements in $\mathbb{F}$.

If $n> m\geq 2$, then $\mathrm{ID}^{*}(L)$ is isomorphic to the Lie superalgebra consisting of  matrices
$$
\left(
  \begin{array}{cc|cc}
      0 &0  & 0    & 0 \\
      A& 0 & C & 0 \\\hline
     0 &  0  & 0 & 0 \\
   D & 0 & B  &  0 \\
  \end{array}
\right)\in \frak{gl}(m|n)      ,
$$
where $A$ is of form (\ref{matricesa}), $B$ is of form (\ref{matricesb}), $C$ is of form
\begin{equation}\label{matricesc4}\left(
  \begin{array}{ccccc}
    c_{1,1}  & \   & \ & \ \\
    c_{2,1}  & c_{1,1}   & \ & \ \\
    \vdots &  \vdots & \ddots &\ \\
   c_{m-1,1}& c_{m-2,1}  & \cdots  &c_{1,1}\\
    \vdots &  \vdots & \ &\vdots \\
   c_{n-1,1}& c_{n-2,1}& \cdots  &c_{n-m+1,2}\\
\end{array}
\right)
\end{equation}
and $D$ is of form
\begin{equation}\label{matricesd4}\left(
  \begin{array}{ccccc}
    d_{1,1}  & d_{1,2}& \    & \ & \ \\
     d_{2,1}  & d_{2,2}& d_{1,2}& \ &\ \\
   \vdots &  \vdots& \vdots& \ddots & \  \\
   d_{m-1,1}& d_{m-1,2}& d_{m-2,2}& \cdots &d_{1,2} \\
  \end{array}
\right)
\end{equation}
with $c_{st}$ and $d_{mn}$ being arbitrary elements in $\mathbb{F}$.

(2) $\mathrm{sdim}\mathrm{ID}^{*}(L)$ attains  the upper bound $\lambda(L^{2}; p, q)$, where $(p, q)$ is the minimal generator number pair of $L/\mathrm{Z}(L)$. Moreover
\begin{equation*}
(p, q)=\left\{
               \begin{array}{ll}
                 (1, 1), & m>n=1\\
                 (2, 0), & n>m=1\\
                 (2, 1), & m\geq 2, n\geq2\\
                 \end{array}
       \right.
\end{equation*}

\end{theorem}

\begin{proof}
 {\textit{Case 1:   $m> n=1$.}} Then
$$
\mathrm{sdim}L^{2}=(0, m-1), \ \mathrm{sdim}\mathrm{Z}(L)=(1, 1).
$$
It is easy to see that $\mathrm{sdim}(L/\mathrm{Z}(L))/(L/\mathrm{Z}(L))^{2}=(1, 1)$.
By Proposition \ref{pro1}, one sees that $(1, 1)$ is the minimal generator number pair of $L/\mathrm{Z}(L)$ and
\begin{equation}\label{case1}
\mathrm{\lambda}(L^{2}; 1, 1)=(m-1, m-1).
\end{equation}
Clearly, an even linear transformation of $L$ is an $\mathrm{ID}^{*}$-superderivation if and only if its matrix with respect to basis (\ref{q2}) is of form
$$
\left(
  \begin{array}{cc|cc}
       &  & \    & \ \\
      &  & \ & \ \\\hline
    \ &   \ & 0 & 0 \\
   \ & \ & B  &  0 \\
  \end{array}
\right)      ,
$$
where $B$ is of form (\ref{matricesb}). Hence, ${\dim}\mathrm{ID}^{*}_{\bar{0}}(L)=m-1$. Similarly, an odd linear transformation of $L$ is an $\mathrm{ID}^{*}$-superderivation if and only if its matrix with respect to  basis (\ref{q2}) is of form
$$
\left(
  \begin{array}{cc|cc}
    \   & \ &    &  \\
    \  & \ &  & \\\hline
    0 &   0& \ & \ \\
   D& 0& \  & \ \\
  \end{array}
\right)      ,$$
where $D$ is of form (\ref{matricesd1}). Therefore ${\dim}\mathrm{ID}^{*}_{\bar{1}}(L)=m-1$. Then by (\ref{case1}), we have
$$
\mathrm{sdim}\mathrm{ID}^{*}(L)=\mathrm{\lambda}(L^{2}; 1, 1).
$$

\textit{Case 2:   $n> m=1$.} Then
$$
\mathrm{sdim}L^{2}=(n-1, 0), \ \mathrm{sdim}\mathrm{Z}(L)=(1, 1).
$$
It is easy to see that $\mathrm{sdim}(L/\mathrm{Z}(L))/(L/\mathrm{Z}(L))^{2}=(2, 0)$.
By Proposition \ref{pro1}, one sees that $(2, 0)$ is the minimal generator number pair of $L/\mathrm{Z}(L)$ and
\begin{equation}\label{2eq3}
\mathrm{\lambda}(L^{2}; 2, 0)=(2n-2, 0).
\end{equation}
Clearly, an even linear transformation of $L$ is an $\mathrm{ID}^{*}$-superderivation if and only if its matrix with respect to basis (\ref{q2}) is of form
$$
\left(
  \begin{array}{cc|cc}
      0 &0  & \    & \ \\
      A& 0 & \ & \ \\\hline
     \ &  \  & \ &  \\
   \ & \ & \  &   \\
  \end{array}
\right)      ,
$$
where $A$ is of form (\ref{matricesa}). Hence, ${\dim}\mathrm{ID}^{*}_{\bar{0}}(L)=2n-2$. Similarly, an odd linear transformation of $L$ is an $\mathrm{ID}^{*}$-superderivation if and only if it is $0$. Therefore ${\dim}\mathrm{ID}^{*}_{\bar{1}}(L)=0$. Then by (\ref{2eq3}), we have
$$
\mathrm{sdim}\mathrm{ID}^{*}(L)=\mathrm{\lambda}(L^{2}; 2, 0).
$$

\textit{Case 3:  $m\geq n\geq 2$.} Then
$$\mathrm{sdim}L^{2}=(n-1, m-1),\ L/\mathrm{Z}(L) \cong L^{n-1, m-1}.
$$
It follows from  Corollary \ref{coro1} that $(2, 1)$ is the minimal generator number pair of $L/\mathrm{Z}(L)$ and
\begin{equation}\label{case3}
\mathrm{\lambda}(L^{2}; 2, 1)=(2n+m-3, 2m+n-3).
\end{equation}
Clearly, an even linear transformation of $L$ is an $\mathrm{ID}^{*}$-superderivation if and only if its matrix with respect to basis (\ref{q2}) is of form
$$
\left(
  \begin{array}{cc|cc}
    0    & 0 & \    & \ \\
    A  & 0 & \ & \ \\\hline
    \ &   \ & 0 & 0 \\
   \ & \ & B  &  0 \\
  \end{array}
\right)      ,
$$
where $A$ is of form (\ref{matricesa})
and $B$ is of form (\ref{matricesb}). Hence, ${\dim}\mathrm{ID}^{*}_{\bar{0}}(L)=2n+m-3$. Similarly, an odd linear transformation of $L$ is an $\mathrm{ID}^{*}$-superderivation if and only if its matrix with respect to  basis (\ref{q2}) is of  form
$$
\left(
  \begin{array}{cc|cc}
    \   & \ & 0    & 0 \\
    \  & \ &  C& 0\\\hline
    0 &   0& \ & \ \\
   D& 0& \  & \ \\
  \end{array}
\right)      ,$$
where $D$ is of form (\ref{matricesd3}) and $C$ is of form (\ref{matricesc3}). Therefore ${\dim}\mathrm{ID}^{*}_{\bar{1}}(L)=2m+n-3$. Then by (\ref{case3}), we have
$$
\mathrm{sdim}\mathrm{ID}^{*}(L)=\mathrm{\lambda}(L^{2}; 2, 1).
$$

\textit{Case 4:  $n> m\geq 2$.} Then
$$\mathrm{sdim}L^{2}=(n-1, m-1),\ L/\mathrm{Z}(L) \cong L^{n-1, m-1}.
$$
It follows from  Corollary \ref{coro1} that $(2, 1)$ is the minimal generator number pair of $L/\mathrm{Z}(L)$ and
\begin{equation}\label{1eq3}
\mathrm{\lambda}(L^{2}; 2, 1)=(2n+m-3, 2m+n-3).
\end{equation}
Clearly, an even linear transformation of $L$ is an $\mathrm{ID}^{*}$-superderivation if and only if its matrix with respect to basis (\ref{q2}) is of form
$$
\left(
  \begin{array}{cc|cc}
    0    & 0 & \    & \ \\
    A  & 0 & \ & \ \\\hline
    \ &   \ & 0 & 0 \\
   \ & \ & B  &  0 \\
  \end{array}
\right)      ,
$$
where $A$ is of form (\ref{matricesa})
and $B$ is of form (\ref{matricesb}). Hence, ${\dim}\mathrm{ID}^{*}_{\bar{0}}(L)=2n+m-3$. Similarly, an odd linear transformation of $L$ is an $\mathrm{ID}^{*}$-superderivation if and only if its matrix with respect to  basis (\ref{q2}) is of  form
$$
\left(
  \begin{array}{cc|cc}
    \   & \ & 0    & 0 \\
    \  & \ &  C& 0\\\hline
    0 &   0& \ & \ \\
   D& 0& \  & \ \\
  \end{array}
\right)      ,$$
where $C$ is of form (\ref{matricesc4}) and $D$ is of form (\ref{matricesd4}). Therefore, ${\dim}\mathrm{ID}^{*}_{\bar{1}}(L)=2m+n-3$. Then by (\ref{1eq3}), we have
$$
\mathrm{sdim}\mathrm{ID}^{*}(L)=\mathrm{\lambda}(L^{2}; 2, 1).
$$ The proof is complete.
\end{proof}

For a more detailed description of the first cohomology groups of  model filiform Lie superalgebras with coefficients in their adjoint modules, the reader is  referred to \cite{L}

\end{document}